\documentclass{llncs} 
\usepackage{latexsym}
\usepackage{epsfig}
\usepackage{amsfonts}
\usepackage{amssymb}
\usepackage{amsmath}
\usepackage{graphics}
\usepackage{ulem}
\usepackage[usenames]{color}
\usepackage{colortbl}

\newcommand{\Qu}{\mathbb Q}
\newcommand{\Zed}{\mathbb Z}
\newcommand{\Nat}{\mathbb N}

\newcommand{\F}{\mathbb {F}}

\newtheorem{maincond}{Hypothesis}{\bfseries}{\itshape}

\newcommand{\x}{{\mathbf{x}}}
\newcommand{\bk}{{\mathbf{k}}}

\begin{document}

\pagestyle{empty}

\mainmatter

\title{Reachability problems for PAMs \thanks{This research is supported by EPSRC grant
 “Reachability problems for words,
matrices and maps” 
(EP/M00077X/1)}}

\titlerunning{Lecture Notes in Computer Science}

\author{Oleksiy Kurganskyy\inst{1}
\and Igor Potapov\inst{2}
%\thanks{Supported by EPSRC grant  (EP/M00077X/1).} 
}

\authorrunning{Oleksiy Kurganskyy and Igor Potapov}

\institute{
Institute of Applied Mathematics and Mechanics, 
NAS of Ukraine 
%\\
%\email{kurgansk@gmx.de}
\and
Department of Computer Science,
University of Liverpool, UK 
%\\
%\email{potapov@liverpool.ac.uk}
}
\maketitle
\begin{abstract}
Piecewise affine maps (PAMs) are frequently used as a reference model 
to show the openness of the reachability questions in other systems.
The reachability problem for one-dimentional PAM  is still open even if we define it with only two intervals. 
As the main contribution of this paper we introduce new techniques 
for solving reachability problems based on $p$-adic norms and weights as well as showing 
decidability for two classes of maps. Then we show the connections
between  topological properties for PAM's orbits, reachability problems and 
representation of numbers in a rational base system. Finally we show  a particular instance where
the uniform distribution of the original orbit may not remain uniform or even dense
after making regular shifts and taking a fractional part in that sequence.\\

{\bf Keywords:}  {\it Reachability problems, piecewise affine maps (PAMs), $\beta$-expansion, $p$-adic analysis}
\end{abstract}          
\section{Introduction}
The simplification of real programs shows that there is a number of quite basic models/fragments 
for which we have fundamental difficulties in the design of verification tools. One of them
is the model of iterative map that appears in many different contexts, including
discrete-event/discrete-time/hybrid systems, qualitative biological models,  
chaos-based cryptography, etc \cite{Aswani2007,BB04,BT00,SODA2015}.

The one-dimensional affine piecewise iterative map is a very rich mathematical object and at the
same time one of the simplest dynamical system producing very complex and sensitive effects. 
A function $f$: $ \mathbb{Q} \rightarrow \mathbb{Q}$ is a one-dimensional piecewise-affine map (PAM) if $f$ is of the form 
$f(x)=a_i  x+b_i$ for $x \in X_i$ where all coefficients $a_i$,$b_i$  and the extremities of a finite number of 
 bounded intervals $X_i$ are rational numbers.
Let us consider the sequence of iterations starting from a rational point $x$: $x$, $f(x)$, $f^2(x)=f(f(x))$, and so on. 
The reachability in PAM is a problem to decide  for a given $f$ and two rational points $x$ and $y$
whether $y$ is reachable from $x$. In other words, is there an $n\in \mathbb{N}$ such that $f^n(x)=y$?

The decidability of the reachability problem for one dimensional piecewise-affine map 
is still an open problem, which is related to other challenging questions in the
theory of computation, number theory and linear algebra \cite{Ben-Amram,KCG94,K2001}.
This model plays a crucial role in the recent research about verification of hybrid systems \cite{Asarin,AS02}, timed automata \cite{Asarin} control systems \cite{BBKPT01,BBKT01}, representation of numbers in a rational base ($\beta$-expansions) \cite{Renyi,Sidorov}, discounted sum automata \cite{TDSP2015}.
In particular PAM is often used as a reference model to show the openness of the reachability questions in other systems.
It also has a very natural geometrical interpretation as pseudo-billiard system \cite{PAMS2008}
and Hierarchical Piecewise Constant Derivative (HPCD) system  \cite{AS02}.
The reachability problem for one-dimensional PAM  is still open even if we define it with only two intervals  \cite{Asarin,AS02,RP14_Bell,RP14_Bournez}. 
%
%
%Defining $f$ with rational functions\cite{PAMS2008},  combination of the following elementary functions $ \{x^ 2 , x^ 3 , x^{1/2}, 
%x^{1/3},2x,x + 1,x -1\}$  in dimension one \cite{PAMS2008}  or piecewise-linear functions in dimension two  \cite{AS02}
%leads to undecidability of the problem. 
%So when such systems are endowed with an additional computational power the reachability becomes undecidable and little 
%simplification can make the system to be predictable, see \cite{Asarin,AS02,RP14_Bell,RP14_Bournez}.
%In dimension one the problem is also open if $f$ is replaced by polynomials or linear rational function (i.e. of the form $f(x)=\frac{a_i 
%\cdot x+ b_i}{c_i \cdot x+ d_i}$).

The primary goal of this paper is to demonstrate new approaches for solving reachability problem in PAMs, connecting reachability
questions  with topological properties of maps 
and widening connections with other important theoretical computer science problems.
First, we show new techniques for decidability of the reachability problem in PAMs based on
$p$-adic norms and weights. We illustrate these techniques showing decidability of two classes of PAMs.
The algorithm in Theorem~\ref{th1} solves point to point reachability problem
for two-interval injective PAM under the assumption that a PAM has 
bounded invariant densities. While our numerical experiments shows that the sequence of 
invariant densities converge to smooth functions it is not yet clear 
whether it holds for all PAMs or if not whether this property can be algorithmically checked.

Following the proposed approach based on $p$-adic weights in 
Theorem~\ref{th2} we define another fragment of PAMs for which the reachability problem is decidable.
In particularly we remove the condition on bounded invariant densities and injectivity of piecewise-affine map 
and consider a PAM $f$ with a constraint on linear coefficients in affine maps.
This class of PAMs is also related to encoding of rational numbers 
in the rational base ($\beta$-expansions). The decidability of the point-to-point problem for this 
class is shown in Theorem~\ref{th2}  and decidability of point-to-set problem 
for the same class can give someone an answer to the open problem related to $\beta$-expansions.
 
Then we establish the connections of topological properties for PAM's orbits with
reachability problems and representation of numbers in a rational base system.
We show that the reachability problems for above objects tightly connected to 
questions about distribution of the fractional parts in the generated sequences
and moreover about distribution of the fractional part after regular shifts.

%We also interpret problems in more general framework where 
% a sequence is hitting dynamical intervals in  $\mathbb {R} / \mathbb {Z}$ and 
%finally show that in a particular instance 
%the uniform distribution of original orbits in maps may not remain uniform or even dense
%when taking a fractional parts after regular shifts.

\section{Preliminaries and Notations}

In what follows we use traditional denotations $ \mathbb {N}$, $ \mathbb {Z} $,
$ \mathbb {Z} ^ + = \{0,1,2, \ldots \} $,  $ \mathbb {P} $, $ \mathbb {Q} $ and $ \mathbb {R }$ 
for sets of natural, integers, positive integers, primes, rational and real numbers, respectively.
Let us denote by  $S^1=\mathbb{Q}/\mathbb{Z}$ the unit circle which consists only rational numbers.
By $\{x\}$, $\left\lfloor x\right\rfloor$ and $\left\lceil x\right\rceil$
we denote the fractional part
\footnote{It will be clear from the context if brackets are used in other conventional ways, for example, to indicate a set of numbers.}
 of a number, floor and ceiling functions.  
%
% upper and lower rounding to the nearest integer number for $x$

Let $Y$ be a set of numbers  and $x$ is a single number, then we define their addition and multiplication as follows:
 $Y+x=x+Y=\{x+y|y\in Y\}$ and $xY=Yx=\{xy|y\in Y\}$. The application of a 
function $f:X\rightarrow Y$ to a set  $X'\subseteq X$ is defined as  $f(X')=\{f(x)|x\in X'\}$.
If $f\subseteq X\times Y$ is a nondeterministic map, i.e.
$f:X\rightarrow 2^Y$ and $x\in X$, $X'\subseteq X$,  we define $f(x)=\{y|(x,y)\in f\}$ and $f(X')=\bigcup_{x\in X'}{f(x)}$.

{\bf $p$-adic norms and weights:}
Let us consider an arbitrary finite set of prime numbers
$\mathbb {F} = \{p_1, p_2, \ldots, p_k \} \subset \mathbb {P}$ 
in ascending order and define the product of prime numbers from $\mathbb {F}$ by  
$m = p_1 p_2 \ldots p_k .$
Let $x$ be a positive rational number that can be represented by primes from a set $\mathbb {F}$. Then its prime factorization is
 $x = \prod_ {p \in \mathbb {F}} {p ^ {\alpha_p}} $ , where $\alpha_p \in \mathbb {Z} $, $p \in \mathbb {F}$.

Any nonzero rational number $x$ can be represented by  $x=(p^{\alpha_p}r)/s$,  
where $p$ is a prime number, $r$ and $s$ are integers not divisible by $p$, and $\alpha_p$ 
is a unique integer. The $p$-adic norm of $x$ is then defined by 
$|x|_p=p^{(-\alpha_p)}$.  
%
%
%The $p$-adic norm of $x$ is denoted by $|x|_p $  and 
The $p$-adic weight of $x$  is defined as  $\left \| x \right \|_p = \log_p (| x | _p) $, i.e. $\left \| x \right \|_p = - {\alpha}_p$. 
The following properties of  $p$-adic weights are directly follows from the properties of $p$-adic norm:
\begin{equation}
\|x\|_p=\|y\|_p\Rightarrow\|x+y\|_p\le\|x\|_p.
\end{equation}
\begin{equation}
\|x\|_p<\|y\|_p\Rightarrow\|x+y\|_p=\|y\|_p,
\end{equation}
\begin{equation}
\|x\cdot y\|_p=\|x\|_p+\|y\|_p,	
\end{equation}
\begin{equation}
\|x^r\|_p=r\|x\|_p,	
\end{equation}
If there is a prime $p \notin \mathbb {F}$ such that $\left \| x \right \|_p> 0$, then we define $\left \| x \right \|_m = + \infty$, 
otherwise  $ \left \| x \right \|_m = \max \limits_ {p \in F} {\left \| x \right \|_p}$.

By {\sl $m$-weight} and { \sl $m$-vector-weight} of $x$ in respect to a set $\mathbb {F}$ we denote 
$\left \| x \right \|_m$ and $ (\left \| x \right \|)_m$=$(\left \| x \right \|_{p_1}, \left \| x \right \|_{p_2}, \ldots, \left \| x \right \|_{p_k})^T$ respectively.
Informally speaking the $m$-weight of a number $x$ (if $\left \| x \right \|_m> 0$) is the number of digits after the decimal point in the representation of $x$ 
in base $m$, i.e. $x$ can be written as $x = y \cdot m^ {- \| x \|_m}$, where $y$ is an integer
which the last digit in the $m$-ary representation is non zero.
If $\left \| x \right \|_m \le 0$, then $x$  is an integer number.

Without loss of generality let us consider from now on only such $x \in X$ for which $\| x \|_m <+ \infty $. 
Alternatively if $ \| x \|_m = + \infty$, it is enough to change the set $ \mathbb{F} = \{p_1, \ldots , p_k \} $
in order to fulfill the requirements of $ \| x \|_m <+ \infty$.

\begin{lemma}\label{lemma0}
For all  rational $ x \in [0,1]$ with an upper bound $a\in \mathbb {Z}^+$ on $ \left \| x \right \|_m$, i.e. $ \left \| x \right \|_m <a $,  there is a lower bound $b \in \mathbb{Z}$ based on $a$ and $\mathbb{F}$
such that   $\left\|x\right\|_p\ge b$ for all $p\in\mathbb{P}$. 
\end{lemma}
\begin{proof}
%Suppose that the conditions of the lemma holds. 
Let us denote two sums of weights:
$\alpha=-\sum_{p\in\mathbb{P},\|x\|_p\le 0}\|x\|_p$
and
$\beta=\sum_{p\in\mathbb{F},\|x\|_p\ge 1}\|x\|_p$.
Assuming that $2$ is the smallest possible prime number and $p_k$ is the largest number in $\mathbb{F} $, we have the following inequality: 
%\[
$\frac{2^{\alpha}}{p_k^{ka}}\le\frac{2^{\alpha}}{p_k^{\beta}}\le x \le 1$.
%\] 
%\noindent
Then $-\alpha\ge -ka{\log_2p_k}$ 
and if $b=-\alpha$ we have $\left\|x\right\|_p\ge b$ for all $p\in\mathbb{P}$. \qed 
\end{proof}

\begin{corollary}\label{cor1}
For any $a\in\mathbb{Z}$ there is only a finite number of rational $x\in [0,1]$ for which $\left\|x\right\|_m<a$.
\end{corollary}

{\bf Reachability problem for PAMs:}
%\begin{definition} 
We say that $f: S^1 \to S^1$  is a one-dimensional piecewise affine map (PAM) whenever $f$ is of the form
$f \left(x \right) = a_ix + b_i$, where  $a_i, b_i \in \mathbb{Q} 
\Leftrightarrow x \in X_i $,  $S^1 = X_1 \cup X_2 \cup \ldots \cup X_l$ and 
where $\{X_1,\ldots,X_l\}$ is a finite family of disjoint (rational) intervals. 
If the intervals are not disjoint we call it {\sl non-deterministic piecewise affine map} and by default
a piecewise-affine mapping is understood to be deterministic.
%\end{definition}
%
The derivative $f'$ of a PAM $f$ we define as $f'(x)=a_i$ for $x\in X_i$, $1\le i\le l$.

If $f$ is deterministic PAM, an orbit (trajectory) of a point $x$ is denoted by 
$O_f(x)$ and will be  understood either as a set $O_f(x)=\{f^i(x)|i\in\mathbb{Z^+}\}$
or as a sequence, i.e.  $O_f(x):\mathbb{Z^+}\rightarrow S^1$, $O_f(x)(i)=f^i(x)$.
We also define that a point $y$ is reachable from $x$ if $y\in O_f(x)$.
%whenever there exists a finite trajectory starting at $x$ and arriving to $y$.

%\begin{definition} 
In general the reachability problem for PAM can be defined as follows. 
Given a PAM $f$,  $x\in S^1$ and $Y\subseteq S^1$, decide whether the intersection $Y\cap O_f(x)$ is empty.
If $Y$ is a finite union of intervals, we name the reachability problem  as {\sl point-to-set (interval)} problem. 
If $Y$ is a one element set (i.e. a single point), the reachability problem is known as {\sl point-to-point} reachability.
\footnote{Also in a similar way it is possible to define set-to-point and set-to-set reachability problems.}
%\end{definition}

In this paper we only consider one-dimensional PAMs and by
 the {\sl reachability problem for PAM} we understand point-to-point reachability 
and explicitly state the type of the problem when we need to refer to other reachability questions.
Note that the point-to-interval reachability can be reduced to point-to-point reachability problem by
extending a map with a few intervals in which the current value is just sequentially deleted. 
It works for all PAMs but may not preserve the properties and the form of the original map.

Generally speaking the piecewise-affine mapping does not need to be defined as $f: X \rightarrow X $, where $X = S^1 = [0,1)$. However if the set $ X $ is a union of any finite number of bounded intervals we always can scale it to $S^1$.
If $f:X\rightarrow X$ is such that $X\neq [0,1)$, and $X\subseteq [a,b)$, then by applying 
conjugation $h(x)=\frac{x-a}{b-a}$ the original reachability problem for $f$ is reduced to 
the reachability problem for the mapping $g=h\circ f\circ h^{-1}$ from  $[0,1)$ to $[0,1)$.
Moreover the interest to PAMs as $f:S^{1}\rightarrow S^{1}$ is also motivated by their use 
in the research of chaotic systems.

\section{Decidability using $p$-adic norms}

It is well know in dynamical systems research that due to complexity of orbits in iterative maps
it is less useful, and perhaps misleading, to compute  the orbit of a single point and
it is more reasonable to approximate the statistics of the underlying dynamics
\cite{Nonlinearity2000,IEEE2002}.
This information is encoded in the so-called {\sl invariant measures}, which specify the probability to observe 
a typical trajectory within a certain region of state space and their corresponding {\sl invariant densities}.

Let us consider a density as an ensemble of initial starting points (i.e. initial conditions). The action of the
dynamical system on this ensemble is described by the Perron-Frobenius operator. 
The ensembles which are fixed under the linear Perron-Frobenius operator 
 is known as 
invariant densities or in other words, they are eigenfunctions with eigenvalue 1 \cite{Nonlinearity2000}.

Formally under an ensemble $A$ we understand an enumerated set (sequence) of points in phase space.
%,
%i.e.  $A$ is a set of points which are not necessarily obtained as an orbit of a map.
With ensemble we can associate the distribution function and the density function.
% of distributed points.
%The density distribution is obtained from the distribution function taking the derivative.
Let $I$ be a set of points.  We denote by $F^{{A}}_I(n)=\left| \{i\in\mathbb{Z^+}|i\le n, A(i)\in I\}\right|$ the number of elements in the sequence $A$ which belong to the set $I$ and which indexes are less or equal $n$.
The distribution function of the ensemble $A$ is defined as 
$\Phi_{A}(x)=\lim_{n\rightarrow\infty}\frac{F^{A}_{(-\infty,x)}(n)}{n}$,  if the limits exist. The density function $\phi_A$ of the ensemble $A$ is defined as $\phi_A(x)=\Phi'_A(x)$.

Suppose
% initially 
given an ensemble $A_0$ with density $\phi_0$. If we apply PAM $f$ to each point of the ensemble, 
we get a new ensemble $A_1$ with some density distribution $\phi_1$. 
%In this case 
We say that the function $\phi_1$ is obtained from $\phi_0$ using 
the Frobenius-Perron or transfer operator, which we denote by $ L_f $.
It is known that $$\phi_1 (x) = L_f (\phi_0) (x) = \sum_ {y \in f ^ {- 1 } (x)} {\frac {\phi_0 (y)} {\left | f '(y) \right |}}.$$
If $\phi_1=\phi_0$ we say that $\phi_0$ is a $f$-invariant density function or an eigenfunction of the transfer operator $L_f$.
 
%\begin{definition}
%A transfer operator (Perron-Frobenius operator) $L_f$ for a piecewise map $f$ is defined on 
%a function $g:S^1\rightarrow \mathbb{R}$  as follows:
% $L_f(g)(x)=\sum_{y\in f^{-1}(x)}{\frac{g(y)}{f'(y)}}$ .
% \end{definition}

%So for any sequences of ensembles $A_1, A_2, \ldots, A_m, \ldots $ we can associate a sequence of distribution functions $\Phi_1, 
%\Phi_2, \ldots, \Phi_m, \ldots $ and the sequence of their distribution densities $\phi_1, \phi_2, \ldots , \phi_m, \ldots $.

We prove that if for an injective PAM $f$ there exists an invariant bounded density function then the reachability problem for $f$ is decidable.

\begin{lemma}\label{lemma1}
Let $f$ be an injective PAM, 
%$L_f$ be the transfer operator for  $f$ 
and
$\phi$ be a $f$-invariant density function.
% of $f$
%eigenfunction with the eigenvalue $1$ of the transfer operator $L_f$.
If there are $K_{min}>0$ and $K_{max}<+\infty$ such that for
any $x$ from the domain of $f$ the following inequality holds:
$K_{min}<\phi(x)<K_{max}$, then for an arbitrary segment of the orbit
$x_1,x_2,\ldots x_{n+1}$, where $x_{i+1}=f(x_i)$,
we have 
$\frac{K_{min}}{K_{max}}\le|c_1\cdot c_2\cdot\ldots\cdot c_n|\le \frac{K_{max}}{K_{min}}$, where $c_i=a_j$ if $x_i\in X_j$.
\end{lemma}

\begin{proof}
Let $\phi$ be an eigenfunction of the Perron-Frobenius operator for an injective PAM $f$. Then injectivity of $f$ and the 
fact that $y = f(x)$ implies that $\phi (y) = \frac {\phi (x)} {\left | f '(x) \right |}$. We denote $f '(x_i)$ by $c_i$. Then $\phi(x_{n+1})=\frac{\phi(x_1)}{\left|c_1\cdot c_2\cdot\ldots\cdot c_n\right|}$  and 
$\left|c_1\cdot c_2\cdot\ldots\cdot c_n\right|=\frac{\phi(x_1)}{\phi(x_{n+1})}$.
Now  we can bound $\left|c_1\cdot c_2\cdot\ldots\cdot c_n\right|$ by 
$\frac{K_{min}}{K_{max}}$ and $\frac{K_{max}}{K_{min}}$.
\end{proof}

\begin{theorem}\label{th1}
Given an injective PAM $f$ with two intervals and 
the existence of a $f$-invariant density function $\phi$
such that there are $ K_ {min}> 0 $ and $ K_ {max} <+ \infty $ and 
the following inequality holds $ K_ {min} <\phi (x) <K_ {max}$ for all $x$ from the domain of $f$. Then the reachability problem for $ f $ is decidable.
\end{theorem}
\begin{proof}
A piecewise-affine map $f$ with two intervals $X_1$ and $X_2$ is defined as follows: 
$f(x) = a_ix + b_i $, if $ x \in X_i$. Note that we only consider PAMs over rational numbers where 
a starting point as well as all coefficients and borders of intervals are rational numbers.
Let us consider an arbitrary pair of points $x,y\in X$ such that $y\in O_f(x)$ and
the sequence of points $x_1,x_2,\ldots,x_{n+1}$ from the orbit  $O_f(x)$ 
such that  $x_1=x$, $x_{n+1}=y$, $x_{i+1}=f(x_i)$, $1\le i\le n$.

If it would be possible to find a computable upper bound $M$ 
on the length $n$ of such reachability sequence  from $x$ to $y$
(based on $K_{min},K_{max} $ $ a_1 $, $ a_2 $, $ b_1 $, $ b_2 $, $ x $, $ y $)
we could solve the reachability problem by considering only initial segment of 
the reachability path of length less or equal to $M$.

In general the existence of such bound is not obvious, however 
if we can show that in the path $x_1,x_2,\ldots,x_{n+1}$ where $x_1=x$ and  $x_{n+1}=y$
the  $p$-adic weights $\|x_i\|_m$ for all $i\in\{1,\ldots,n+1\}$ are bounded by some value $M_1$
then we can restrict $M$ as follows:  $ n <M <m ^ {M_1}$
as $\|x_i\|_m$  is the number of decimal places in the representation of $ x_i $ in base $m$.
In particular if there is a orbit's segment of 
length that is greater than $m ^ {M_1}$ and which 
consists of numbers with $p$-adic weights less than $M_1$
then it always contains two identical numbers and the orbit loops.

So in order to prove a computable upper bound on $ \| x_i \|_m $ it
is sufficient to prove that for any $p \in \F$ 
$p$-adic weights (i.e. logarithmic norms) $ \| x_i \|_p $ are bounded from above by a computable number  $M_2$ for all $i \in \{1, \ldots, n + 1 \}$. 
Note that by the definition of the logarithmic $m$-norm, $ M_2 = M_1 $.

Let us define a number $h=\max\{\|b_{1}\|_m+1,\|b_{2}\|_m+1,\|x_1\|_m,\|x_{n+1}\|_m\}$.  
Let us take an arbitrary $p \in \F$ and select a subsequence 
$x_j$, $x_{j+1}$, $\ldots,x_{r}$ in the sequence $x_1$, $x_2$, $\ldots$, $x_{n+1}$
such that its elements $x\in\{x_{j+1}, x_{j+2}, \ldots,x_{r-1}\}$ have the following property
$\|x\|_p > h$, but at the same time $\|x_{j}\|_p\le h$ and $\|x_{r}\|_p\le h$.
For simplicity, but without loss of generality, we assume that $j = 1$, $r = n + 1$ and $ \| x_1 \|_p = \| x_{n + 1} \|_p = h$.

Later we either  will find the computable upper bound on $n$ or will show a computable bound $M_2$  on $p$-adic weights
based on $K_ {min}, K_ {max}, a_1, a_2$ and $h$. As stated above, it is enough for proving the theorem.

Let us define $x_{i+1}=f(x_i)=c_ix_i+d_i$, where $c_i\in \{a_1,a_2\}$, $d_i\in\{b_1,b_2\}$.
From the properties (1) and (2) of the $p$-adic weights and the fact that 
$\|x_i\|_p>\|b_j\|_p$, $i\in\{1,2,\ldots,n\}$, $j\in\{1,2\}$ follows that $\|x_{i+1}\|_p=\|x_i\|_p+\|c_i\|_p$.
This implies 
%\[
$\|x_{n+1}\|_p=\|x_1\|_p+\sum_{i=1}^{n}\|c_i\|_p=\|x_1\|_p+\sum_{i=1}^{2}\alpha_i\|a_i\|_p$
%\]
for some non-negative numbers $\alpha_1$ and $\alpha_2$, where $\alpha_1+\alpha_2=n$. 
Taking into account that  $\|x_1\|_p=\|x_{n+1}\|_p$ we have
%\[
$\sum_{i=1}^{2}\alpha_i\|a_i\|_p=0.$
%\]

Let $r$  be  the greatest common divisor of $\alpha_1$ and $\alpha_2$.
We define $\alpha=\alpha_1/r$ and $\beta=\alpha_2/r$, i.e. $\alpha$ and $\beta$ 
are smallest non-negative integers such that  $\alpha \|a_1\|_p+\beta \|a_2\|_p=0$.
Thus, we obtain 
%\[
$\sum_{i=1}^{2}\alpha_i\|a_i\|_p=r(\alpha \|a_1\|_p+\beta \|a_2\|_p).$
%\]
By Lemma~\ref{lemma1} we have:
%\[
$\frac{K_{min}}{K_{max}}\le|c_1\cdot c_2\cdot\ldots\cdot c_n|\le \frac{K_{max}}{K_{min}}$,
%\] 
i.e. 
%\[
$\frac{K_{min}}{K_{max}}\le|a_1^{\alpha}a_2^{\beta}|^r\le \frac{K_{max}}{K_{min}}$,
%\] 

Now we consider two cases $|a_1^{\alpha}a_2^{\beta}|\neq 1$ and $|a_1^{\alpha}a_2^{\beta}|=1$. 
The first case corresponds to linearly independent columns of the matrix $A_f$, $\operatorname{rank}(A)=2$, and 
the second case to linear dependence, $\operatorname{rank}(A_f)<2$.

Let $|a_1^{\alpha}a_2^{\beta}|\neq 1$, then either $|a_1^{\alpha}a_2^{\beta}|> 1$ or $|a_1^{\alpha}a_2^{\beta}|< 1$.
If $|a_1^{\alpha}a_2^{\beta}|> 1$ then from  $|a_1^{\alpha}a_2^{\beta}|^r\le \frac{K_{max}}{K_{min}}$ follows that  $r\le\frac{\ln\frac{K_{max}}{K_{min}}}{\ln |a_1^{\alpha}a_2^{\beta}|}$.  On the other hand if  $|a_1^{\alpha}a_2^{\beta}|< 1$,
then from $|a_1^{\alpha}a_2^{\beta}|^r\ge \frac{K_{min}}{K_{max}}$ follows $r\le \frac{\ln\frac{K_{min}}{K_{max}}}{\ln |a_1^{\alpha}a_2^{\beta}|}$.

Now suppose that $|a_1^{\alpha}a_2^{\beta}|= 1$. 
The only interesting case is when $a_1 \neq 1 $ and $a_2 \neq 1$ as the cases where  $|a_1|=1$, $|a_2|=1$ or both
correspond to the trivial case of piecewise-affine mapping (i.e. with no more than one linear factor).

Without loss of generality, we assume that  $a_1> 1$ and $a_2<1$. Note also that $\|a_2\|_p=-\frac{\alpha}{\beta}\|a_1\|_p$.
Let us now consider an arbitrary subsequence of consecutive points $x_1x_2\ldots x_{j+1}$, $j\le n$ of the original reachability path. 
Assuming that $\alpha_1$, $\alpha_2$ such that $|c_1c_2\ldots c_j|=|a_1|^{\alpha_1}|a_2|^{\alpha_2}$ we have that
%\[
$$\|x_{j+1}\|_p=\|x_1\|_p+\alpha_1\|a_1\|_p+\alpha_2\|a_2\|_p=(\alpha_1-\alpha_2\frac{\alpha}{\beta})\|a_1\|_p.$$
%\]
On the other hand from Lemma~\ref{lemma1} we know that:
%\[
$\frac{K_{min}}{K_{max}}\le|a_1|^{\alpha_1}|a_2|^{\alpha_2}\le \frac{K_{max}}{K_{min}}.$
%\]
Since $|a_1^{\alpha}a_2^{\beta}|= 1$,then  $|a_2|=|a_1|^{-\frac{\alpha}{\beta}}$ and
\[
\frac{K_{min}}{K_{max}}\le|a_1|^{\alpha_1-\alpha_2\frac{\alpha}{\beta}}\le \frac{K_{max}}{K_{min}}.
\]
Taking into account assumption that $|a_1|>1$, we get
%\[
$\alpha_1-\alpha_2\frac{\alpha}{\beta}\le\frac{\ln \frac{K_{max}}{K_{min}}}{\ln |a_1|}.$
%\]

Now it follows that:
\[
\|x_{j+1}\|_p=\|x_1\|_p+(\alpha_1-\alpha_2\frac{\alpha}{\beta})\|a_1\|_p\le h+\left(\frac{\ln \frac{K_{max}}{K_{min}}}{\ln |a_1|}\right)\|a_1\|_m.
\]
Thus, we have shown a computable upper bound for $p$-adic weights  of the orbital elements that can reach a point $y$. 
Finally, in view of provided reasoning, we shown that the reachability problem for this type of PAMs is decidable. \qed
\end{proof}
%\begin{proof}
%See in Appendix.
%\end{proof}

%The theorem can be applied for a larger class of maps if more information would be known 
%about the convergence of distribution functions and density of distributions.
%Let us call an ensemble $A$ to be statistically fixed with respect to $f$, if $\Phi_A = L_f (\Phi_A)$.
%E.g if someone can show that in injective PAM 
 %all statistically fixed ensembles have identical distribution functions then Theorem~{th1} can be 
%directly applied to show decidability in injective one-dimensional PAMs. 

The theorem can be applied for a larger class of PAMs if more information would be known about the convergence of density functions under the action of the Perron-Frobenius operator.
Let us call an ensemble $A$ to be statistically fixed with respect to $f$, if $\phi_A = L_f (\phi_A)$. E.g. if someone can show that in injective PAM 
 all statistically fixed ensembles have identical distribution functions then Theorem~{\ref{th1}} can be
%directly 
applied to show decidability in injective PAMs.

Following the proposed approach based on $p$-adic weights 
we define another fragment of PAMs for which the reachability problem is decidable.
In particularly we remove the condition on eigenfunction of the transfer operator and injectivity of piecewise-affine map 
and consider a PAM $f$ with only constraints on linear coefficients in affine maps.
More specifically we require that the powers of prime numbers from 
prime factorizations of linear coefficients should have the same signs (i.e. two sets 
of prime numbers used in nominator and denominator are disjoint).
Let us denote for a PAM $f$ a matrix $A_f$ with values $(a_{ji})$, where $a_{ji} = \| a_i \|_{p_j}$, $1 \le i \le l$, $ 1 \le j \le k$. 
The rank of $A_f$ is  denoted by $\operatorname{rank} (A_f)$.

\begin{theorem}\label{th2}
The reachability problem for a PAM $f$ is decidable
if every row of a matrix $A_f$  contains values of the same sign, 
(i.e.  $a_{ji}\cdot a_{j'i}\geq 0$, 
for all $i,j$ such that
$1\le i\le l$, $1\le j,j'\le k$).
\end{theorem}
\begin{proof}
Let us consider a PAM $f$
 of the form 
$f(x)=a_i  x+b_i$ for $x \in X_i$ where all coefficients $a_i$,$b_i$  and the extremities of a finite number of 
 bounded intervals $X_i$ are rational numbers.
Let us define $h=\max\{\|b_{1}\|_m,\|b_{2}\|_m,\ldots,\|b_l\|_m\}$.  
The condition of the theorem means  that for any prime $p\in\mathbb{F}$ 
all linear coefficients of the map $f$ 
have non-zero $p$-adic weights of the same sign.

In this case, if $p$-adic weights of linear coefficients of $f$ are non-negative, 
then for any $x\in X$ from $\|x\|_p>h$ follows that  $\|f(x)\|_p\geq\|x\|_p$ 
and therefore 
 $\|f(x)\|_m\geq\|x\|_m$ (i.e. $m$-adic weight  does not decrease).
If $p$-adic weight of linear coefficients of the mapping are negative, then for any
$x\in X$ we have $\|f(x)\|_p\le\max\{\|x\|_p,h\}$.

Thus, in the sequence of reachable points for an orbit of a map $f$ either all points of the orbit have  $m$-adic weights
bounded from above  by $h$, then we have a cyclic orbit, or from some moment 
when $m$-adic weight of a reachable point exceeds $h$  it does not decrease and 
again, either orbit loops or $m$-adic weight increases indefinitely.

Thus, in order to decide whether $y$ is reachable, i.e.  $y \in O_f(x)$, it
is sufficient to start generating a sequence of reachable points in the orbit  $O_f(x)$ and wait for one of the events, 
where either
1) a point in the orbit is equal to $y$  {\sl ( $y$ is reachable )}, or 
2) the orbit will loop and $y \notin O_f(x) $   {\sl ( $y$ is not reachable )}, or 
3) a point $x'$ is reachable,  such that $\|x' \|_m>\max \{h, \| y \|_m \} $, 
and then $ y \notin O_f(x) $   {\sl ( $y$ is not reachable )}.  \qed
\end{proof}

\begin{definition}
A piecewise affine mapping $f: S^1 \to S^1$  is complete if for a set of disjoint intervals $S^1 = X_1 \cup X_2  \cup \ldots \cup X_n$,
$f(X_i) = S^1$ for any $i=1..n$.
\end{definition}

\begin{definition}
Let be $F:\mathbb{R}\rightarrow \mathbb{R}$ is the lifting of a continuous map  $f:S^1\rightarrow S^1$ on $\mathbb{R}$, 
i.e.  $f(\{x\})=\{F(x)\}$. 
Then by the degree $\deg(f)$ of a map $f$ we denote the number $F(x+1)-F(x)$, which is independent from 
the choice of the point $x$ and the lifting $F$.
\end{definition}

\begin{corollary}
The reachability problem for complete piecewise affine mappings with two intervals \footnote{In particularly  the continuous piecewise affine mapping of degree two} is decidable.
\end{corollary}
\begin{proof}
The condition of a piecewise affine map with two intervals $f: S^1 \to S^1$ to be complete means that 
$S^1 = X_1 \cup X_2 $ and $f (X_1) = f (X_2) = S^1$.
Thus, if  $X_1=\left[ 0,\frac{m}{n}\right]$ and $X_2=\left[\frac{m}{n}, 1\right)$, 
then $f\left(x\right)=a_1x+b_1$,
where $ a_1 = \pm \frac {n} {m}$, when $ x \in X_1 $, and $ f \left (x \right) = a_2x + b_2 $, where $ a_2 = \pm \frac {n} {n-m} $ at $ x \in X_2 $, $ m, n \in \mathbb {N} $, $\gcd (m, n) = 1 $. It is clear that $n$, $m$, $n-m$ are relatively prime.
So the conditions of Theorem~\ref{th2} are satisfied. \qed
\end{proof}

\section{PAM representation of $\beta$-expansions}

Given a rational non-integer $\beta>1$ and the number $x\in [0,1] $. 
The target discounted-sum 0-1 problem \cite{Dagstuhl_Games_2015,TDSP2015} is defined as follows:
%\begin{problem}
{\it Is there a sequence $ w: \mathbb {N} \rightarrow \{0,1 \} $ of zeros and ones such that 
$x = \sum_{i = 1}^ {\infty} {w (i) \frac {1} {\beta^ i}}$.} 
%\end{problem}

For any $x \in S^1 $, there exists
$\beta$-expansion $w: \mathbb {N} \rightarrow \{0,1, \ldots, \lceil \beta \rceil-1 \}$ 
such that $x= \sum_{i = 1}^{\infty} {w (i) \frac {1} {\beta^ i}}$. 
If $w (i) \in \{0,1\}$ we call it {\bf $(0,1)-\beta$-expansion}.
Therefore, when $\beta \le 2$ the answer to the target discounted-sum problem is always positive. 
Therefore, the only interesting case is when $\beta>2$. We denote $D = \{0,1, \ldots, \lceil \beta \rceil-1 \}$. 
Then the minimal and maximal numbers, which are representable in the basis $\beta$ 
with digits from the alphabet $D$, are 
$min = \sum_ {i = 1}^{\infty} {0 \frac {1} {\beta^ i}} = 0 $ and $max = \sum_ {i = 1}^ {\infty} {(\lceil \beta \rceil-1) 
\frac {1} {\beta^ i}} = \frac {\lceil \beta \rceil-1} {\beta-1} $. When $\beta> 2$ then $max$ is always less then two. 
Let us denote by $X_d$ the interval $[\frac{min + d} {\beta}, \frac {max + d} {\beta})$ for each $d \in D$.
If $\beta$ is not an integer number then two intervals  $X_d$ and $X_ {d + 1}$ intersect. 
Also taking into account that $max <2$, then the intervals $X_d$ and $X_{d + 2}$ have no common points. 
Finally from the above construction we get the next lemma:
\begin{lemma}\label{lemma_intervals}
If $\beta>2$ and $\beta$ is rational/non-integer number:
%\begin{center}
$X_d\cap X_{d+1}\neq\emptyset$, $d<\lceil\beta\rceil-1$; 
$X_d\cap X_{d+2}=\emptyset$, $d<\lceil\beta\rceil-2$;
 $[min,max)=\cup_{d\in A}{X_d}$.
%\vspace*{-0.5cm}
%\end{center}
%\begin{enumerate}
%	\item $X_d\cap X_{d+1}\neq\emptyset$, $d<\lceil\beta\rceil-1$;
%	\item $X_d\cap X_{d+2}=\emptyset$, $d<\lceil\beta\rceil-2$;
%	\item $[min,max)=\cup_{d\in A}{X_d}$.
%\end{enumerate}
\end{lemma}
\begin{proposition} For any $\beta$-expansion  there is
a non-deterministic PAM where a symbolic dynamic of visited intervals
(i.e. a sequence of symbols associated with  intervals)  from 
an initial point $x_0$ 
corresponds to its representation in base $\beta$.
\end{proposition}
\begin{proof} 
Let us define the piecewise affine mapping $f\subseteq [min,max)\times [min,max)$ as follows
 $f=\{(x,\beta x-d)|x\in X_d, d\in D\}$. 
It directly follows from this definition that $f(X_d)=[min,max)$. 
%
%\vspace{-0.3cm}
%\vspace{-0.6cm}
\begin{figure}[htb]
	\centering
		\includegraphics[width=0.55\textwidth]{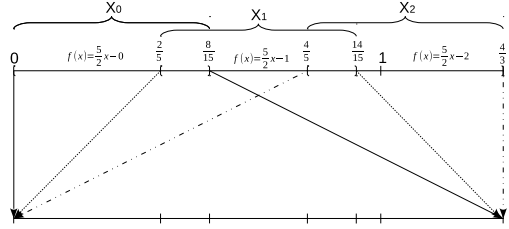}
				\caption{A non-deterministic PAM for $\frac{5}{2}$-expansion}
	\label{fig:pam5-2}
\end{figure}
%\vspace{-0.6cm}
%\vspace{-0.5cm}
%
Let us consider an orbit $f^i(x)=x(i)$, $i \in \mathbb{Z^+}$.
We say that $d_i \in D$ such that  $x(i) \in X_{d_i}$. 
Then for any $n\in\mathbb{N}$ $min<\left|\beta^nx-d_1\beta^{n-1}-d_2\beta^{n-2}-\ldots-d_n\beta^0\right|<max$,
and in other form $\frac{min}{\beta^{n}}<\left|x-\sum_{i=1}^{n}(d_i\frac{1}{\beta^i})\right|<\frac{max}{\beta^{n}}$. 
So $\left|x-\sum_{i=1}^{n}(d_i\frac{1}{\beta^i})\right|\rightarrow 0$, $n\rightarrow\infty$, 
and therefore $x=\sum_{i=1}^{\infty}(d_i\frac{1}{\beta^i})$. 
Let us consider it in other direction. Let 
$x=\sum_{i=1}^{\infty}(d_i\frac{1}{\beta^i})$, then the sequence  $\x(i)$, where $\x(0)=x$, $\x(i+1)=\beta \x(i)-d_i$,
is the orbit of  $x$ in PAM $f$. 
Let us name the constructed map as the {\bf $\beta$-expansion PAM}. 
\qed \end{proof}

The nondeterministic $\beta$-expansion can be translated into deterministic maps 
corresponding to {\sl greedy} and {\sl lazy} expansions as follows:
\begin{definition}
A function $f:[min,max)\rightarrow [min,max)$ is the greedy $\beta$-expansion PAM
if the domain $[min,max)$ is divided on intervals $X'_{d}$, $d\in\{0,1,\ldots,\lceil\beta\rceil-1\}$
such that $X'_{\lceil\beta\rceil-1}=X_{\lceil\beta\rceil-1}$,
$X'_{d-1}=X_{d-1}-X_{d}$,  $d\in\{1,2,\ldots,\lceil\beta\rceil-1\}$
and $f(x)=\beta x-d$ iff  $x\in X'_d$.
\end{definition}

Since $X_d=[\frac{min+d}{\beta},\frac{max+d}{\beta})$ then 
$X'_d=[\frac{min+d}{\beta},\frac{min+d+1}{\beta})=[\frac{d}{\beta},\frac{d+1}{\beta})$, $d\in\{0,1,\ldots,\lceil\beta\rceil-2\}$
and the length of the interval $X'_d$ is equal to $\frac{1}{\beta}$, $d<\lceil\beta\rceil-1$.
\vspace{-0.3cm}
\begin{figure}
	\centering
		\includegraphics[width=0.49\textwidth]{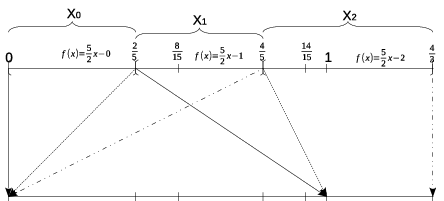}
		\includegraphics[width=0.49\textwidth]{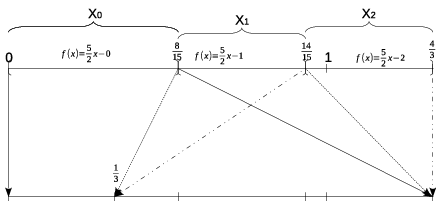}
				\caption{Deterministic greedy (on the left) and lazy (on the right) $\frac{5}{2}$-expansion PAM}
	\label{fig:pam5-2-2}
\end{figure}
\vspace{-0.4cm}
%
%\begin{figure}
%	\centering
%		\includegraphics[width=0.48\textwidth]{pam5-2-2.eps}
   %             \includegraphics[width=0.48\textwidth]{pam5-2-greedy.eps}
	%			\caption{Deterministic greedy $\frac{5}{2}$-expansion PAM with unreachable points  which are $>1$.}
%	\label{fig:pam5-2-2}
%\end{figure}
%\vspace{-0.6cm}
%
%
\begin{definition}
A function $f:[min,max)\rightarrow [min,max)$ is the lazy $\beta$-expansion PAM, if the domain $[min,max)$
is divided into intervals $X''_{d}$, $d\in\{0,1,\ldots,\lceil\beta\rceil-1\}$, 
such that $X''_{0}=X_{0}$, $X''_{d}=X_{d}-X_{d-1}$,  $d\in\{1,2,\ldots,\lceil\beta\rceil-1\}$, 
and $f(x)=\beta x-d$ iff $x\in X''_d$.
\end{definition}
%\vspace{-0.6cm}
%\begin{figure}
%	\centering
%		\includegraphics[width=0.48\textwidth]{pam5-2-lazy-1.eps}
%		\includegraphics[width=0.48\textwidth]{pam5-2-lazy-2.eps}
%				\caption{Deterministic lazy $\frac{5}{2}$-expansion PAM withunreachable points  which are $<\frac{1}{3}$.}
%	\label{fig:pam5-2-lazy-1}
%\end{figure}
%\vspace{-0.6cm}
%
%
%
\begin{proposition}
Let  $f$ and $g$ are {\sl greedy} and {\sl lazy} $\beta$-expansion PAM's respectively.  
$f$ and $g$ are (topologically) conjugate by
the homeomorphism  $h:$  $h(x)=h^{-1}(x)=max-x$, i.e. $f=h\circ g\circ h$.
\end{proposition}
  \begin{proof}
The statement holds since $X'_{d}=max-X''_{\lceil\beta\rceil-1-d}$, $d\in\{0,1,\ldots,\lceil\beta\rceil-1\}$ 
%\qed 
\end{proof}

%From above definition we can also derive that if $f$ is the greedy $\beta$-expansion PAM then 
%$f(X_d)=[0,1)$ if $d\in\{0,1,\ldots,\lceil\beta\rceil-2\}$ and if 
%$f$ is the lazy $\beta$-expansion PAM, then $f(X_d)=[max-1,max)$, если $d\in\{1,2,\ldots,\lceil\beta\rceil-1\}$.

We would like to highlight that the questions about
reachability as well as representation of numbers
in rational bases are tightly connected with questions about the density of orbits in PAMs. 
Moreover if the density of orbits are the same for all non-periodic points then 
it may be possible to have a wider application of 
$p$-adic techniques provided in the beginning of the paper. Let us formulate a hypothesis that 
goes along with our experimental simulations in PAMs:

\begin{maincond}\label{condA}
The orbit of any rational point in any expanding deterministic PAM is either finite or dense on the whole domain.
\end{maincond}

\begin{lemma}
Any $(0,1)-\beta$-expansion is greedy.
\end{lemma}
\begin{proof}
Let $f$ be a $\beta$-expansion PAM. Assume that there is a point $x$ and the orbit
$\x(i)$, where $\x(0)=x$, $\x(i+1)=\beta \x(i)-d_i$ in the map $f$ such that 
$d_i\in \{0,1\}$ for all $i \in \mathbb{N}$ and the orbit does not correspond to the $\beta$ greedy expansion of
$x$. 

The intersection of intervals $X_0$ and $X_1$ is an interval $X_{01}=[\frac{\min+1}{\beta},\frac{\max}{\beta})$.
Applying a map $y=\beta x$ to $X_{01}$ we see that  $X_{01}$ is 
scaled into $[\min +1, \max)=[1,\frac{\lceil\beta\rceil-1}{\beta-1})$. 
The interval $[1,\frac{\lceil\beta\rceil-1}{\beta-1})$ does not have any common points with $X_0$
as the point $1$ lies on the right side of the left border of the interval 
$X_2=[\frac{\min+2}{\beta},\frac{\max+2}{\beta})$ and by Lemma~\ref{lemma_intervals} 
 $X_{i}\cap X_{i+2}=\emptyset$. 

Note that when $x>\frac{1}{\beta-1}$ we have $\beta x-1 >x$. Let us assume that for some $i$
$\x(i)\in X_{01}$ and $\x(i+1)=\beta \x(i)-0$, i.e. we did not followed a greedy expansion
and therefore $\x(i+1)\in [1,\frac{\lceil\beta\rceil-1}{\beta-1})$.
Then $\x(i+2)=\beta \x(i+1)-1>\x(i+1)$ and $\x(i+2)\notin X_0$, etc.
In this case starting from $\x(i+1)$ there is 
monotonically increasing sequence of orbital points in the interval $X_1$.
So points in such orbit should eventually leave the interval 
$X_1$ and reach $X_d$, where $d>1$. This gives us a contradiction with the original assumption. 
\qed\end{proof}

\begin{corollary}
Since the greedy expansion can be expressed by a deterministic map then $(0,1)-\beta$-expansion 
is  unique and greedy.
\end{corollary}

\begin{theorem}
If Hypothesis~\ref{condA} holds then
a non-periodic $(0,1)-\beta$-expansion does not exist.
\end{theorem}
\begin{proof}
Any $(0,1)-\beta$-expansion can be constructed by expanding \footnote{I.e. with linear coefficients that are greater than $1$} deterministic greedy $\beta$-expansion PAM.
If the orbit of a rational point in greedy $\beta$-expansion PAM is non-periodic, then by  Hypothesis~\ref{condA}
it should be dense and therefore should intersect all intervals and cannot provide 
$(0,1)-\beta$-expansion. \qed
\end{proof}

\begin{theorem}
If Hypothesis~\ref{condA} holds then for any rational number its deterministic $\beta$-expansion is either
eventually periodic or it contains all possible patterns (finite subsequences of digits)  from $\{0,1,\ldots,\lceil\beta\rceil-1\}$.
\end{theorem}
\begin{proof} 
The statement is obvious as Hypothesis~\ref{condA} implies that the orbit is either periodic or it is dense and the 
dense orbit is visiting all intervals. \qed 
\end{proof}

It looks that the point-to-interval problem is harder than the point-to-point reachability problem for the expanding PAMs, 
as for example  Theorem~\ref{th2} gives an algorithm 
for the point-to-point reachability problem in the $\beta$-expansion PAMs,  but not for the point-to-interval reachability that is equivalent to
 the $\beta$-expansion problem.

Note that in the $\beta$-expansion PAMs all linear coefficients are the same, so the density of the orbit correspond to the 
density of the following sequence $\x(n)=f^n(x_0)$, where $f(x)=\{\beta  x\}$. 
For example when $\beta = \frac{5}{2}$ and $x_0=1$ we get the sequence:
\begin{center}
 $\{\frac{5}{2}\}, \{ \frac{5}{2} \{\frac{5}{2}\} \},  \{ \frac{5}{2}  \{ \frac{5}{2} \{\frac{5}{2}\} \} \}, \ldots$ 
\end{center}
The question about the distribution of a similar sequence
 $\{\frac{3}{2}\}$,  $\{\frac{3^2}{2^2}\}$,   $\{\frac{3^3}{2^3}\}$, ...
, where the integer part is removed once after 
taking a power of a fraction (for example 3/2) is known as Mahler's $3/2$ problem, that is
a long standing open problem in analytic number theory.

\section{Density of orbits and its geometric interpretation}

It is well known that $\x(n)=\{\alpha n\}$, where $\alpha$ is an irrational number,
has an uniform distribution. Let us give some geometric interpretation of the orbit density.
Consider the Cartesian plane with the y-axis  $x$ and and the x-axis $y$ (just swapping their places).
Now let us divide the set of lines $x = n$, $n \in \Nat$, by integer points on the segments of the unit length.
The set of points $(y,x)$, where $m \le y <m + 1$, $x = n$, i.e. the interval $[m, m + 1 ) \times n$ 
on the line $x = n ,$ will be denoted by $S_{m,n}$, $m\in\Zed^+$, $n \in \Nat$.  In other words, $S_{m,n} = (m + [0,1)) \times n$.
Let $I$ be an interval such that $I \subseteq [0,1)$ and by $I_{m, n} $  let us denote 
the set $(m + I) \times n$.

%\vspace{-0.7cm}
\begin{figure}
	\centering
        \includegraphics[scale=0.31]{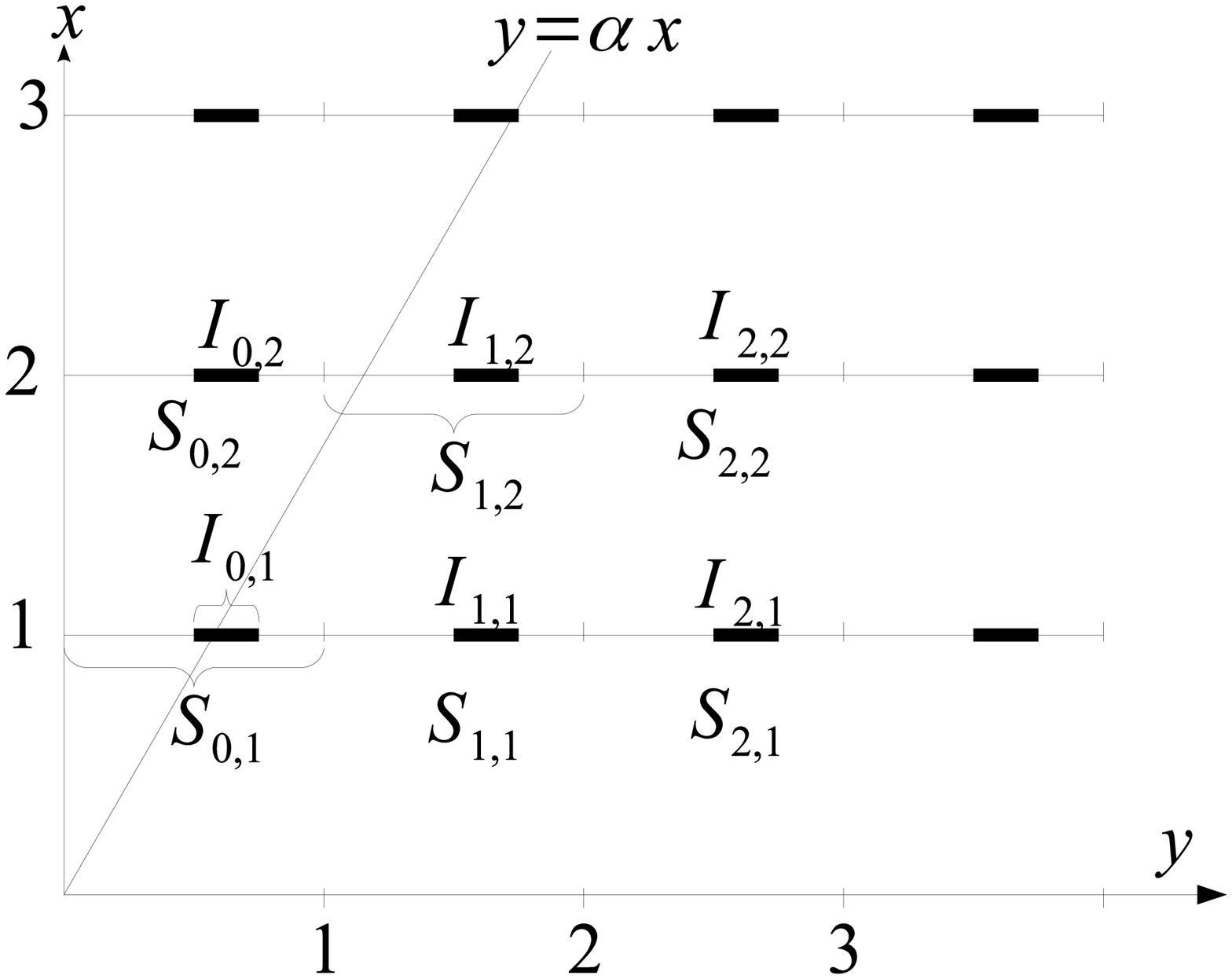}
  	\includegraphics[scale=0.31]{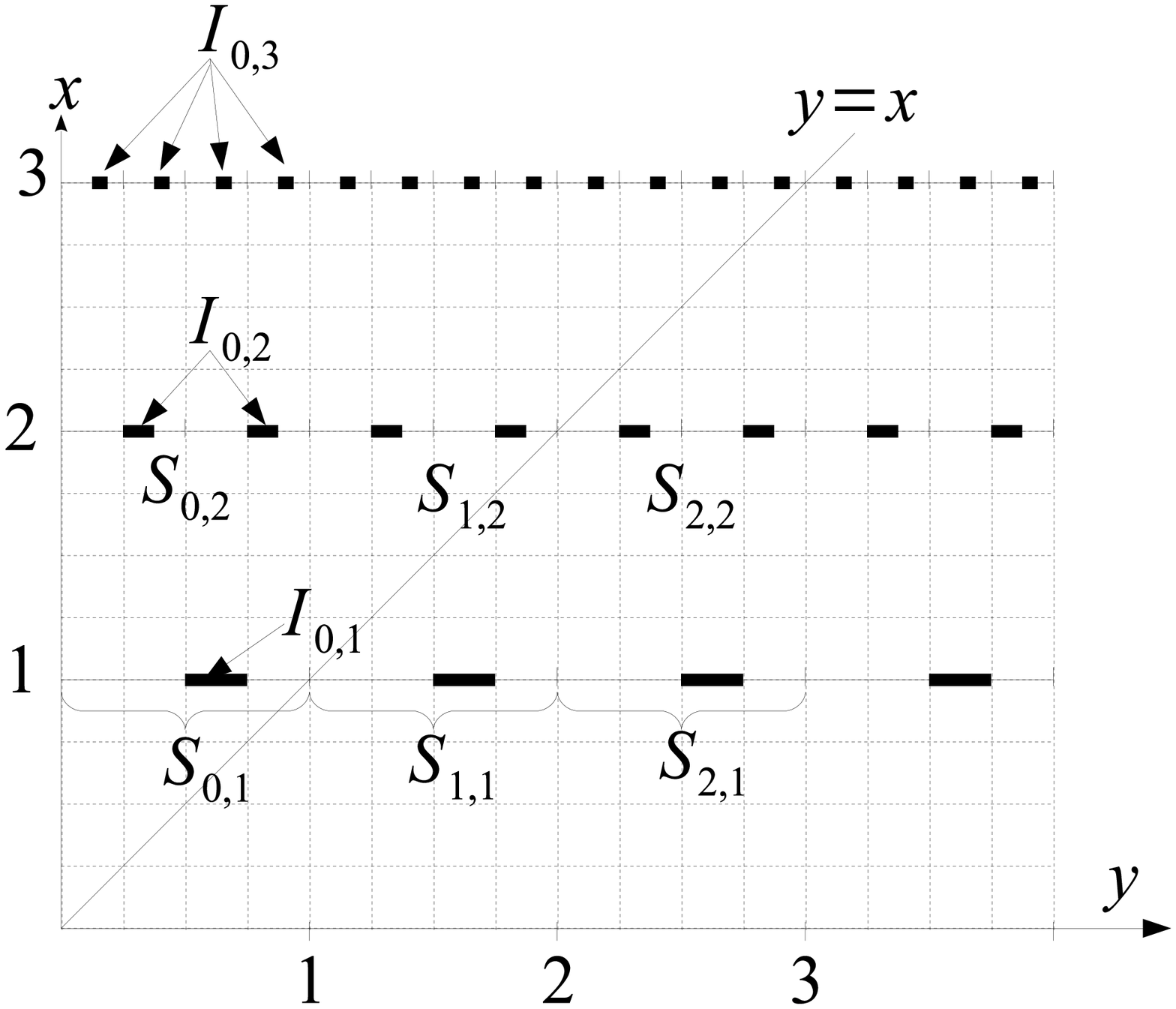}
				\caption{Left: An example for two sets $S_{m,n}$ and $I_{m,n}$; Right: A dynamic interval $I(n)$.}
	\label{fig:din1}
\end{figure}
%\vspace{-0.4cm}

Two points of the plane are defined to be equivalent if they belong to a same line passing through the origin.
We call $\alpha$ as homogeneous coordinate of a point $(y,x)$ if $y = \alpha x$.
By $H(I)$ we denote the set of homogeneous coordinates of all points from $\bigcup \limits_ {m\in\Zed^+, n \in \Nat} {I_{m, n}} $.
The sequence $\x (n) = \{\alpha n \} $ is dense in $[0,1)$ if and only if for any interval $I \subseteq [0,1)$ 
there are $m$ and $n$, such that the line $ y = \alpha  x $ intersects the set $ I_{m, n}$. 
It is known that $[0,1)-H(I) \subseteq \Qu$ for any interval $I \subseteq [0,1)$, 
i.e.  for any irrational $\alpha>0$ the line $y = \alpha x $ intersect the set $ \bigcup \limits_ {m\in\Zed^+, n \in \Nat} {I_ {m, n}}$. 
Moreover in the case of irrational factors it is known that the frequency of occurrence of  $\x (n) = \{\alpha n \}$ 
in the interval $I$ is equal to its length.

In some sense the interval $I$, in the above example, can be named as {\sl static} because it does not change in time
$n$. However in order to study and describe previously mentioned problems such as 
the target discounted-sum problem,  PAMs reachability problems, 
the Mahler's $3/2$ problem we require the notion of ``dynamic intervals``.

Let $\x$  be a sequence of numbers from $[0,1)$. What is the distribution of a sequence $\x '(n) = \left \{p^ {\bk (n)} \x (n) \right \}$, where $\mathbf {k}: \mathbb {Z^+} \rightarrow \mathbb {Z ^ +}$ is a non-decreasing sequence? 
For example, if $\bk (n) = n-1$ and  the number $\x (n) $ has in the base $p$ the following form
$\x(n)=0.a_{n1}a_{n2}\ldots a_{nn}a_{n,n+1}\ldots $,
then $\x'(n)=0.a_{nn}a_{n,n+1}\ldots$.

%Connections with analytic number theory...

%Let us given a sequence  $\mathbf{x}:\mathbb{Z^+}\rightarrow S^1$ and an interval  $I\subseteq S^1$. 
%By  $F^{\mathbf{x}}_I(n)=\left| \{i\in\mathbb{Z^+}|i\le n, x(i)\in I\}\right|$
%we denote the number of elements in the sequence which belong to the interval $I$ and which indexes are less than $n$.
%The distribution function of the sequence $\mathbf {x}$ is defined as 
%$\Phi_{\mathbf{x}}(y)=\lim_{n\rightarrow\infty}\frac{F^{x}_{(0,y)}(n)}{n}$,
% $ 0 \le y <1 $,  if the limits exist.
%
%The target discounted-sum problem, piecewise-affine maps reachability problems, the Mahler's $3/2$ problem and the study of 
%$\beta$-expansions for rational numbers are tightly connected to the question: ``Given  $p \in \mathbb{N}$ 
%and a sequence  $\mathbf{x}$.
%What can be stated about the distribution of the sequence 
%$\mathbf{x'}:\mathbb{Z^+}\rightarrow S^1$, 
%where $\mathbf{x'}(i)=\left\{p^{\mathbf{k}(i)}\mathbf{x}(i)\right\}$, where $\mathbf{k}:\mathbb{Z^+}\rightarrow \mathbb{Z^+}$ 
%is a non-decreasing sequence''?

Let us assume that $I\subseteq S^1$ and $\mathbf{k}:\mathbb{Z^+}\rightarrow \mathbb{Z^+}$  is a non-decreasing sequence,
 $p\in\mathbb{N}$. Now we define ``dynamical intervals'' as an 
evolving infinite sequence  
$\mathbf{I}(1)$, $\mathbf{I}(2)$, $\mathbf{I}(3)$, $\ldots$ :
%We start with the interval $\mathbf{I}(1)=I$ and then on every discrete step $i$ the union of intervals $\mathbf{I}(i)$ is modified 
%into $\mathbf{I}(i+1)$ as follows:
%\[
$$\mathbf{I}(1)=I, 
\mathbf{I}(n)=
\bigcup_{j=0}^{p^{\mathbf{k}(n)}-1}{\frac{I+j}{p^{\mathbf{k}(n)}}}.$$
%\]

By $F^{\mathbf{x}}_{\mathbf{I}}(n)=\left| \{i\in\mathbb{Z^+}|i\le n, \mathbf{x}(i)\in \mathbf{I}(i)\}\right|$  
we denote a function representing 
a frequency of hitting dynamical interval $\mathbf{I}$ by the sequence $\mathbf{x}$.
In contrast to $F^{\mathbf {x}}_{I}(n)$ which only 
counts the number of hittings to a fixed interval $I$, our new function $F^{\mathbf{x}}_{\mathbf{I}}$ counts 
the number of  hittings when both points and intervals are changing in time.

\begin{proposition}
The following equation holds: $F^{\mathbf{x'}}_{I}(n)=F^{\mathbf{x}}_{\mathbf{I}}(n)$.
\end{proposition}

%\vspace{-0.7cm}
%\begin{figure}
%	\centering
%		\includegraphics[scale=0.3]{din2.eps}
%				\caption{A dynamic interval $I(n)$.}
%	\label{fig:din2}
%\end{figure}
%\vspace{-0.4cm}

%\begin{example}
The phenomenon that significant digit distribution in real data are not accruing randomly known as Benford's Law.
For example the sequence $p^1$, $p^2$, $p^3$,.. satisfies Benford's Law, under the condition that $\log_{10} p$ 
is an irrational number, which is a consequence of the {\sl Equidistribution theorem} (proved separetly  by  Weyl, Sierpinski and Bohl). 
The Equidistribution theorem states that the sequence
$\{\alpha \}$, $\{2\alpha \}$, $\{3\alpha \}$, $\ldots$
is uniformly distributed on the circle    $\mathbb {R} / \mathbb {Z}$ , when $a$ is an irrational number.
It gives us the fact that each significant digit of numbers in ($p^n$) sequence will correspond to the interval $\mathbb {R} / \mathbb {Z}$ 
and the length of the interval related to the frequency for each appearing digit.

However the question about the distribution of the sequence $\{(3/2)^n\}$  is different in the way that
it is not about the distribution of the first digits of $3^n$ in base $2$, i.e.  not about the distribution of the sequence $\frac{3^n}{2^{\lceil n\log_2{3}\rceil}}$, but related to the sequence of digits
after some shift of the number  $\frac{3^n}{2^{\lceil n\log_2{3}\rceil}}$ corresponding to the multiplication by a power of $2$. 

So in the above notations the distribution of numbers in the sequence $\mathbf{x'}(n) =\{(3/2)^n\}$ corresponds to the 
$F^{\mathbf{x'}}_{I}(n)$ for the logarithmic (Benford's law) distributed sequence
$\x(n)=\frac{3^n}{2^{\lceil n\log_2{3}\rceil}}$, $p=2$ and $\bk(n)={\lceil n\log_2{3}\rceil-n}$.
 %In other words we are interested in distribution of  $\bk(n)$ digits in the sequence $\x(n)$.
%\begin{figure}
%	\centering
%		\includegraphics[width=0.90\textwidth]{pam5-2.eps}
%				\caption{A non-deterministic PAM for $\frac{5}{2}$-expansion}
%	\label{fig:pam5-2}
%\end{figure}
%\end{example}

Now we will show that even if the sequence $\{\alpha \}$, $\{2\alpha \}$, $\{3\alpha \}$, $\ldots$
is uniformly distributed on the circle $\mathbb {R} / \mathbb {Z}$,
%, when $\alpha$ is an irrational number (following the Equidistribution theorem), i.e. when $\bk(i)=1$,
the irrationality of $\alpha$ is not enough to guarantee uniform distribution or even density 
of the sequence $\x'(n)$ on the circle corresponding to the linear shifts $\bk(n)=n$.

\begin{theorem}
Let us define  $\alpha=\sum_{i=1}^{\infty}{\frac{1}{2^{\Delta_i}}}$ where
 $\Delta_1=1$, $\Delta_{i+1}=2^{\Delta_i}+\Delta_i$, $i\geq 1$ (http://oeis.org/A034797). 
Then for all $n\in \mathbb{N}\cup\{0\}$ 
the sequence $\{2^nn\alpha\}$ is not dense in the interval $[0,1]$
and $\{2^nn\alpha\}<\frac{1}{2}$.
\end{theorem}
\begin{proof}
Let us consider a sequence $\Delta$, which initial elements are 
$\Delta_1=1$, $\Delta_2=3$, $\Delta_3=11$, $\Delta_4=2059$ etc.

{
%\tiny
\[
\begin{array}{cccccccccccccccc}
       &  &   &\Delta_1&&\Delta_2&&&&&&&&\Delta_3&&\\	
\mathbf{x}(0) &=&0,&0&0&0&0&0&0&0&0&0&0&0&0&\ldots\\	
\mathbf{x}(1)  &=&0,&\fbox{1}&0&1&0&0&0&0&0&0&0&1&0&\ldots\\	
\mathbf{x}(2)  &=&0,&0&\fbox{1}&0&0&0&0&0&0&0&1&0&0&\ldots\\	
\mathbf{x}(3)  &=&0,&0&1&\fbox{1}&0&0&0&0&0&0&1&1&0&\ldots\\	
\mathbf{x}(4)  &=&0,&1&0&0&\fbox{0}&0&0&0&0&1&0&0&0&\ldots\\	
\mathbf{x}(5)  &=&0,&1&0&1&0&\fbox{0}&0&0&0&1&0&1&0&\ldots\\	
\mathbf{x}(6)  &=&0,&1&1&0&0&0&\fbox{0}&0&0&1&1&0&0&\ldots\\	
\mathbf{x}(7)  &=&0,&1&1&1&0&0&0&\fbox{0}&0&1&1&1&0&\ldots\\	
\mathbf{x}(8)  &=&0,&0&0&0&0&0&0&0&\fbox{1}&0&0&0&0&\ldots\\	
\mathbf{x}(9)  &=&0,&0&0&0&0&0&0&0&1&\fbox{0}&0&1&0&\ldots\\	
\mathbf{x}(10)&=&0,&0&0&0&0&0&0&0&1&0&\fbox{1}&0&0&\ldots\\	
\mathbf{x}(11)&=&0,&0&0&0&0&0&0&0&1&0&1&\fbox{1}&0&\ldots\\	
\end{array}
\hspace{0.2cm}
\]

\[
\begin{array}{cccccccccccccccc}
\mathbf{x'}(0) &=&0,&0&0&0&0&0&0&0&0&0&0&0&0&\ldots\\	
\mathbf{x'}(1)  &=&&0,&0&1&0&0&0&0&0&0&0&1&0&\ldots\\	
\mathbf{x'}(2)  &=&&&0,&0&0&0&0&0&0&0&1&0&0&\ldots\\	
\mathbf{x'}(3)  &=&&&&0,&0&0&0&0&0&0&1&1&0&\ldots\\	
\mathbf{x'}(4)  &=&&&&&0,&0&0&0&0&1&0&0&0&\ldots\\	
\mathbf{x'}(5)  &=&&&&&&0,&0&0&0&1&0&1&0&\ldots\\	
\mathbf{x'}(6)  &=&&&&&&&0,&0&0&1&1&0&0&\ldots\\	
\mathbf{x'}(7)  &=&&&&&&&&0,&0&1&1&1&0&\ldots\\	
\mathbf{x'}(8)  &=&&&&&&&&&0,&0&0&0&0&\ldots\\	
\mathbf{x'}(9)  &=&&&&&&&&&&0,&0&1&0&\ldots\\	
\mathbf{x'}(10)&=&&&&&&&&&&&0,&0&0&\ldots\\	
\mathbf{x'}(11)&=&&&&&&&&&&&&0,&0&\ldots\\	
\end{array}
\]
}

Let us prove that when $0\leq n\leq \Delta_i$ the inequality $\{2^nn\alpha\}<\frac{1}{2}$
if and only if $\{2^nn\sum_{j=1}^{i}{\frac{1}{2^{\Delta_j}}}\}<\frac{1}{2}$. 
The implication from left to right follows from $\{2^nn\sum_{j=1}^{i}{\frac{1}{2^{\Delta_j}}}\}<\{2^nn\alpha\}$. 
Let us show that it also holds in other direction.
Assume that $\{2^nn\sum_{j=1}^{i}{\frac{1}{2^{\Delta_j}}}\}<\frac{1}{2}$
then  $\{2^nn\sum_{j=1}^{i}{\frac{1}{2^{\Delta_j}}}+x\}<\frac{1}{2}$ for any $0\leq x\leq \frac{1}{2^{\Delta_i+1-n}}$. 
In this case it is enough to show that  $2^nn\alpha-2^nn\sum_{j=1}^{i}{\frac{1}{2^{\Delta_j}}}<\frac{1}{2^{\Delta_i+1-n}}$ holds when $n\leq\Delta_i$.  
In fact we have  that 
$$2^nn\alpha-2^nn\sum_{j=1}^{i}{\frac{1}{2^{\Delta_j}}}=
2^nn\sum_{j=i+1}^{\infty}{\frac{1}{2^{\Delta_j}}}<2^nn\frac{1}{2^{\Delta_{i+1}-1}}=$$
$$=\frac{1}{2^{\Delta_{i+1}-n-\log_2(n)-1}}\leq\frac{1}{2^{\Delta_{i+1}-\Delta_i-\log_2(\Delta_i)-1}}=
=\frac{1}{2^{2^{\Delta_i}-\log_2(\Delta_i)-1}}.$$ 
Finally we have $\frac{1}{2^{2^{\Delta_i}-\log_2(\Delta_i)-1}}<\frac{1}{2^{\Delta_i+1-n}}$ when $n\leq\Delta_i$ and $i>1$.

Now for proving the theorem it is enough to show that for $0\leq n\leq\Delta_i$ the inequality
$\{2^nn\sum_{j=1}^{i}{\frac{1}{2^{\Delta_j}}}\}<\frac{1}{2}$ holds.
When $i=1$ the statement is trivial. Let us assume that it holds for $i=k-1$.
We show now that it holds for $i=k$, i.e. we show that for 
$\Delta_{k-1}\leq n\leq \Delta_{k}$  inequality  $\{2^nn\sum_{j=1}^{k}{\frac{1}{2^{\Delta_j}}}\}<\frac{1}{2}$
holds or taking into account  $\Delta_{k-1}\leq n$
also we have inequality $\{2^nn\frac{1}{2^{\Delta_k}}\}<\frac{1}{2}$. 

%\noindent
For $\Delta_{k-1}\leq n \leq \Delta_{k}-\Delta_{k-1}-1$ we get 
$2^nn\frac{1}{2^{\Delta_k}}\leq 2^{\Delta_{k}-\Delta_{k-1}-1}(\Delta_{k}-\Delta_{k-1}-1)\frac{1}{2^{\Delta_k}}$.
Since  $\Delta_k=2^{\Delta_{k-1}}+\Delta_{k-1}$, then
$2^{\Delta_{k}-\Delta_{k-1}-1}(\Delta_{k}-\Delta_{k-1}-1)\frac{1}{2^{\Delta_k}} =(2^{\Delta_{k-1}}-1)\frac{1}{2^{\Delta_{k-1}+1}}=\frac{1}{2}-\frac{1}{2^{\Delta_{k-1}+1}}<\frac{1}{2}$.

Consider the case when
 $\Delta_{k}-\Delta_{k-1}\leq n \leq \Delta_{k}$,  i.e. $2^{\Delta_{k-1}}\leq n \leq 2^{\Delta_{k-1}}+\Delta_{k-1}$
and define  $m=n-2^{\Delta_{k-1}}$.
We have that 
$0\leq m\leq\Delta_{k-1}$ and $2^nn\frac{1}{2^{\Delta_k}}=2^{m+2^{\Delta_{k-1}}}({m+2^{\Delta_{k-1}}})\frac{1}{2^{2^{\Delta_{k-1}}+\Delta_{k-1}}}=2^{m}({m+2^{\Delta_{k-1}}})\frac{1}{2^{\Delta_{k-1}}}=2^{m}m\frac{1}{2^{\Delta_{k-1}}}+2^m$. 
Therefore $\{2^nn\frac{1}{2^{\Delta_k}}\}=\{2^{m}m\frac{1}{2^{\Delta_{k-1}}}\}$, where $0\leq m\leq \Delta_{k-1}$. 
Now by the induction we have that $\{2^{m}m\frac{1}{2^{\Delta_{k-1}}}\}<\frac{1}{2}$. 
\end{proof}

While the question about the distributions for PAM orbits remains open we have unexpectedly shown that in a very similar system,
operating with irrational numbers, the uniform distribution of original orbits in maps may not remain uniform or even dense
when taking the fractional part after regular shifts. This makes the questions about PAMs 
even more ``mysterious'' as it is not clear whether such property may hold for a
sequence of points generated by PAMs, $\beta$-expansion and Mahler's problem.

\end{document}